\documentclass[10pt,reqno]{amsproc}
\newtheorem{theorem}{\bf{Theorem}}[section] 
\newtheorem{lemma}[theorem]{\bf{Lemma}}     

\newtheorem{corollary}[theorem]{\bf{Corollary}}
\newtheorem{proposition}[theorem]{\bf{Proposition}}
\newtheorem{definition}[theorem]{\bf{Definition}}
\newtheorem{remark}[theorem]{\bf{Remark}}

\title[COMPLETELY SEMI-$\varphi$-MAPS]
 {COMPLETELY SEMI-$\varphi$-MAPS} 


\usepackage{color}

\author[M. B. Asadi]{Mohammad B. Asadi}
\address{School of Mathematics, Statistics and Computer Science,
 College of Science, University of Tehran, Tehran,
  Iran, and \\
   School of Mathematics, Institute for Research in Fundamental Sciences (IPM),
P.O. Box: 19395-5746, Tehran, Iran}

\email{mb.asadi@khayam.ut.ac.ir}

\author[R. Behmani]{Reza Behmani}
\address{\noindent Department of Mathematics, Kharazmi University, 50, Taleghani Ave.,15618, Tehran Iran}

\email{reza.behmani@gmail.com}

\author[A. R. Medghalchi]{Ali R. Medghalchi}
\address{\noindent Department of Mathematics, Kharazmi University, 50, Taleghani Ave.,15618, Tehran IRAN.}

\email{a\_medghalchi@khu.ac.ir}%

\author[H. Nikpey]{Hamed Nikpey}
\address{Department of Mathematics, Shahid Rajaei Teacher Training University, Tehran 16785-136, Iran}
\email{hamednikpey@gmail.com}

\begin{document}
\maketitle

 \begin{abstract}
We introduce completely semi-$\varphi$-maps on Hilbert $C^*$-modules as a generalization of $\varphi$-maps.
   This class of maps provides examples of CP-extendable maps which are not CP-H-extendable, in Skeide-Sumesh's sense.
  Using the CP-extendability of completely semi-$\varphi$-maps, we give a representation
theorem,  similar to Stinespring's representation theorem, for
this class of maps which can be considered as strengthened and
generalized form of Asadi's and Bhat-Ramesh-Sumesh's analogues
   of Stinespring representation theorem for $\varphi$-maps.
     We also define an order relation on the set of all completely semi-$\varphi$-maps and establish
      a Radon-Nikodym type theorem for this class of maps in terms of their representations.
         \end{abstract}

\section{INTRODUCTION}
Asadi in \cite{Asadi} and Bhat-Ramesh-Sumesh in \cite{BRS} gave
a representation theorem for a class of maps on Hilbert $C^*$-modules,
 as a generalization of Stinespring's representation theorem for completely positive maps
on $C^*$-algebras.
  Michael Skeide \cite{skeide} achieved a generalization of Bhat-Ramesh-Sumesh's theorem
  for $\varphi$-maps between Hilbert $C^*$-modules in term of $C^*$-correspondences.

L. Aramba$\check{s}$i$\acute{c}$ \cite{Arambašic}, extended the
representation theory of $C^*$-algebras to Hilbert $C^*$-modules.
     She showed that the set of all representations of a Hilbert $C^*$-module $\mathcal{E}$ are in one to one
      correspondences with the set of all representations of its linking $C^*$-algebra $\mathcal{L}(\mathcal{E})$.
Skeide and  Sumesh \cite{Skeide-Sumesh} introduced CP-extendable
maps between Hilbert $C^*$-modules as maps that can be extended to
a completely positive map  acting block-wise between the
associated (extended or reduced) linking algebras. They
characterized $\varphi$-maps in term of those CP-extendable maps
where the 11-corner of the extension can be chosen to be a
homomorphism, the CP-H-extendable.
                Besides of the studying CPH-semigroups in \cite{Skeide-Sumesh}, they presented a factorization of strictly CP-extendable maps, too.
          Combining Stinespring's representation theorem \cite{St,Paulsen} and
           a result of L. Aramba$\check{s}$i$\acute{c}$ \cite[Proposition 3.1]{Arambašic} imply
            that an operator valued map on a Hilbert $C^*$-module is dilatable if and only if
            it is CP-extendable.

           The above results and facts motivate us to take a closer look at dilatable maps and provide a class of CP-extendable maps which
            are not CP-H-extendable.
            We introduce the class of completely semi-$\varphi$-maps as a generalization of $\varphi$-maps.
             We concentrate on operator valued completely semi-$\varphi$-maps
             and strengthen Skeide-Sumesh' theorem \cite[Theorem 1.3]{Skeide-Sumesh} in operator valued case by showing that completely
             semi-$\varphi$-maps are exactly those CP-extendable maps which its $11$-corner of the extension can be chosen to be a
             unital completely positive map. Using this we strengthen the main result of \cite{Asadi} for completely semi-$\varphi$-maps in Section 3.

                Moreover, we use the minimality conditions  that was introduced in \cite{BRS} for dilation pairs of $\varphi$-maps
                 to introduce the minimal dilation pairs of completely semi-$\varphi$-maps, and  show that two minimal dilation pairs for
                  a given completely semi-$\varphi$-maps on a Hilbert $C^*$-module $\mathcal{E}$ implement unitarily equivalent $*$-representations
                   on the linking $C^*$-algebra $\mathcal{L}_1(\mathcal{E}).$
                 Furthermore, we give two characterizations of completely semi-$\varphi$-maps in terms of
                 their CP-extension and minimal dilation pairs which helps us to construct examples of completely semi-$\varphi$-maps.

In Section 5, we define an order relation on the set of all
completely semi-$\varphi$-maps and provide a Radon-Nikodym type
theorem  for completely semi-$\varphi$-maps in terms of their
dilation pairs. This Radon-Nikodym type theorem for completely
semi-$\varphi$-maps maps strengthen the Joita's result on
$\varphi$-maps on Hilbert $C^*$-modules (c.f. \cite[Theorem
2.15]{J}).

\section{PRELIMINARIES}

 For a right Hilbert $C^*$-module $\mathcal{E}$ over  a
unital $C^*$-algebra $\mathcal A$, the linking $C^*$-algebra of
$\mathcal{E}$ is denoted by $\mathcal{L}(\mathcal{E})$ and defined
as $\mathcal{L}(\mathcal{E}):=\{\begin{bmatrix}
u & x \\
y^* & a
\end{bmatrix} | \ a\in\mathcal{A}, \ u\in\mathbb{K}(\mathcal{E}), \ x, y\in\mathcal{E}\},$ where $\mathbb{K}(\mathcal{E})$ is the set of compact operators on
$\mathcal{E}$.
We consider the unitization of $\mathcal{L}(\mathcal{E})$ as
$\mathcal{L}_{1}(\mathcal{E}):=\{\begin{bmatrix}
u & x \\
y^* & a
\end{bmatrix} | \ a\in\mathcal{A}, \ u \in \mathbb{K}_{1}(\mathcal{E}), \ x,y\in\mathcal{E}\},$
where
$\mathbb{K}_{1}(\mathcal{E})=\mathbb{K}(\mathcal{E})+\mathbb{C}I_{\mathcal{E}},$ $(I_\mathcal{E}$ is
the identity operator on $\mathcal{E}$ and when there is no confusion, it is denoted by $I$,
 for convenience we denote $\lambda I$ by $\lambda$ for every complex scaler $\lambda).$
The smallest operator subsystem of
$\mathcal{L}_{1}(\mathcal{E})$ which contains $\mathcal{A}$ and
$\mathcal{E}$ is denoted by $S_{\mathcal{A}}(\mathcal{E})$ and is
defined as follow
$$S_{\mathcal{A}}(\mathcal{E}):=\begin{bmatrix}
\mathbb{C}I & \mathcal{E} \\
\mathcal{E}^* & \mathcal{A}
\end{bmatrix}=\{\begin{bmatrix}
\lambda & x \\
y^* & a
\end{bmatrix} | \ a\in \mathcal{A}, \ \lambda\in\mathbb{C}, \ x,y\in\mathcal{E}\}.$$

For every natural number $n,$ $\mathbb{M}_n(\mathcal{E})$ with its
natural vector space structures and the following module action and
inner product is a Hilbert $C^*$-module over the $C^*$-algebra
$\mathbb{M}_n(\mathcal{A})$,

(i) $(x_{ij}).(a_{ij}):=(\sum_{k=1}^{n}x_{ik}a_{kj})$ for every
$(a_{ij})\in\mathbb{M}_n(\mathcal{A})$ and
$(x_{ij})\in\mathbb{M}_n(\mathcal{E}),$

(ii) $\langle(x_{ij}),(y_{ij})\rangle:=(\sum_{k=1}^{n}\langle
x_{ki},y_{kj}\rangle)$ for every
$(x_{ij}),(y_{ij})\in\mathbb{M}_n(\mathcal{E}).$

For Hilbert spaces $H,K$, and arbitrary given maps
$\rho:\mathcal{A}\rightarrow\mathcal{B}(H)$,
$\sigma:\mathbb{K}_{1}(\mathcal{E})\rightarrow\mathcal{B}(K)$ and
 $\Psi:\mathcal{E}\rightarrow\mathcal{B}(H,K)$,
 the map $\begin{bmatrix}
u & x \\
y^* & a
\end{bmatrix}\mapsto \begin{bmatrix}
\sigma(u) & \Psi(x) \\
\Psi(y)^* & \rho(a)
\end{bmatrix}$ from $\mathcal{L}_{1}(\mathcal{E})$ into $\mathcal{B}(K \oplus H)$
is denoted by $\begin{bmatrix}
\sigma & \Psi \\
\Psi^* & \rho
\end{bmatrix}$.
Similarly, if $T \in \mathcal{B}(K)$, the
map $\begin{bmatrix}
\lambda & x \\
y^* & a
\end{bmatrix}\mapsto \begin{bmatrix}
\lambda T & \Psi(x) \\
\Psi(y)^* & \rho(a)
\end{bmatrix}$ from $S_{\mathcal{A}}(\mathcal{E})$ into $\mathcal{B}(K \oplus H)$
 is denoted by $\begin{bmatrix}
T & \Psi \\
\Psi^* & \rho
\end{bmatrix}$.

Assume that $\mathcal{E,F}$ are Hilbert $C^*$-modules over $C^*$-algebras $\mathcal{A,B}$ respectively,
$\varphi:\mathcal{A}\rightarrow\mathcal{B}$ is a completely
positive map and $\Phi:\mathcal{E}\rightarrow\mathcal{F}$ is
a linear map, we say

(1) $\Phi$ is a \textit{$\varphi$-map}, if
$\langle\Phi(x),\Phi(y)\rangle=\varphi(\langle x, y\rangle)$, for all
$x,y \in\mathcal{E}$.

(2) $\Phi$ is a \textit{semi-$\varphi$-map}, if
$\langle\Phi(x),\Phi(x)\rangle\leq \varphi(\langle x,x\rangle)$, for all
$x\in\mathcal{E}.$

(3) $\Phi$ is a \textit{completely semi-$\varphi$-map}, if
$\langle\Phi_n(x),\Phi_n(x)\rangle\leq \varphi_n(\langle x,x\rangle)$, for all
$x\in \mathbb{M}_n(\mathcal{E})$ and $n \in \mathbb{N}$.

 If $\mathcal B=\mathcal{B}(H)$ and $\mathcal{F}=\mathcal{B}(H,K)$
for some Hilbert spaces ${H,K},$

(4) $\Phi$ is \textit{non-degenerate}, if
$[\Phi(\mathcal{E})H]=K$.

(5) $\Phi$ is a \textit{$\varphi$-representation or
representation}, if $\Phi$ is a $\varphi$-map and $\varphi$ is a
$*$-representation.

(6) $\Phi$ is \textit{dilatable}, if there exists a representation
$\Psi:\mathcal{E}\to\mathcal{B}(H',K')$ and bounded operators
$V:H\to H'$ and $W:K\to K'$ such that
$$\Phi(x)=W^*\Psi(x)V.$$


\begin{remark} \label{r2.1}
If $\rho:\mathcal{A}\to\mathcal{B}(H)$ is a $*$-representation and
$\Psi:\mathcal{E}\to\mathcal{B}(H,K)$ is a $\rho$-representation,
then there exists a $*$-representation
$\sigma:\mathbb{K}_{1}(\mathcal{E})\rightarrow\mathcal{B}(K)$ such
that for every
 $x,y\in\mathcal{E}$, $\sigma(x\otimes y)=\Phi(x)\Phi(y)^*.$
Consequently,
$\begin{bmatrix}
\sigma & \Psi \\
\Psi^* & \rho
\end{bmatrix}: \mathcal{L}_{1}(\mathcal{E}) \rightarrow \mathcal{B}(K \oplus H)$
is a representation of
$\mathcal{L}_{1}(\mathcal{E})$.
 Conversely, every representation of the linking algebra has this form \cite[Proposition 3.1]{Arambašic}.

 Note that we use a weaker definition for non-degenerate
  operator valued maps on Hilbert $C^*$-modules rather than
  L. Aramba$\check{s}$i$\acute{c}$'s definition \cite[Definition 3.2]{Arambašic}
  for non-degenerate representations on Hilbert $C^*$-modules.
   However, in the case of full Hilbert $C^*$-modules, nondegeneracy of
    the $\rho$-representation $\Psi$ implies that $[\Psi(\mathcal{E})^*K]=H$, and consequently the two definitions coincide
  and $\sigma,$ $\rho$ and also
$\begin{bmatrix}
\sigma & \Psi \\
\Psi^* & \rho
\end{bmatrix}$ are non-degenerate  if and only if $\Psi$ is non-degenerate by \cite[Lemma 3.4]{Arambašic}.
\end{remark}


\begin{remark}
Note that every $\varphi$-map is a completely semi-$\varphi$-map.
Also if we consider $\mathcal{A}$ as a Hilbert
$\mathcal{A}$-module, then every unital completely positive map
$\varphi:\mathcal{A}\rightarrow\mathcal{B}(H)$ is a completely
semi-$\varphi$-map on $\mathcal{A}$. In this case $\varphi$ is a
$\varphi$-map iff $\varphi$ is a $*$-representation. Thus, for
every non-multiplicative unital completely positive map $\varphi,$
there is a completely semi-$\varphi$-map which is not a
$\varphi$-map. At the end of section 4, we provide a
characterization of operator valued completely semi-$\varphi$-maps
which helps us to construct completely semi-$\varphi$-maps which
are not $\varphi$-map.
\end{remark}


\begin{proposition}
 Assume $\varphi:\mathcal{A}\rightarrow\mathcal{B}(H)$ is a pure, unital completely positive map which is not multiplicative.
   Then $\varphi$ is a completely semi-$\varphi$-map but it is not $\tau$-map for any  completely positive map $\tau:\mathcal{A}\rightarrow\mathcal{B}(H).$
\end{proposition}
\begin{proof}
Let $(\rho,K,V)$ be the minimal Stinespring dilation triple for $\varphi.$ Then $\rho:\mathcal{A}\rightarrow\mathcal{B}(K)$ is an irreducible representation for $\mathcal{A}$ \cite[Corollary 1.4.3]{Arv}. If there exists a completely positive map $\tau:\mathcal{A}\rightarrow\mathcal{B}(H)$ such that $\varphi(x)^*\varphi(y)=\tau(a^*b)$ satisfies for every $x,y\in\mathcal{E},$ then $\tau\leq\varphi,$ since $\varphi$ is a completely semi-$\varphi$-map.
 Thus there exists a positive contraction $T\in\rho(\mathcal{A})'$ such that $\tau(a)=V^*T\rho(a)V$ for every $a\in\mathcal{A}$ \cite[Theorem 1.4.2]{Arv}, but $\rho$ is an irreducible representation, thus $\rho(\mathcal{A})'=\mathbb{C}.I_{K},$ so, $T=t.I_{K}$ for some scaler $t\in [0.1].$
Therefore $\varphi(a)^*\varphi(b)=t\varphi(a^*b)$ for every
$a,b\in\mathcal{E},$ which implies $t=1$ and therefore $\varphi$ is
a $\varphi$-map, thus $\varphi$ is multiplicative, which is a
contradiction.
  \end{proof}

\section{COMPLETELY SEMI-$\varphi$-MAPS, CP-EXTENDABILITY AND DILATABILITY}

In the following, we show that each completely semi-$\varphi$-map
$\Phi$ on a Hilbert $C^*$-module implements a completely positive
map on the linking $C^*$-algebra.
 In fact, we show that $\varphi$
and $\Phi$ are corners of a completely positive map on
$\mathcal{L}_{1}(\mathcal{E}).$
 The following lemma can be
obtained by \cite[Lemma 3.1]{Paulsen}.

\begin{lemma}\label{l}
Let $\mathcal{A}$ be a unital $C^*$-algebra and $\mathcal{E}$  a
right Hilbert module over $\mathcal{A}.$
 Then for every
$x\in\mathcal{E}$ and $a\in\mathcal{A}$, $\begin{bmatrix}
1 & x \\
x^* & a
\end{bmatrix}$ is positive if and only if $\langle x,x\rangle_\mathcal{A}\leq a.$

\end{lemma}


\begin{lemma}\label{l3.1}
Let $\varphi:\mathcal{A}\rightarrow\mathcal{B}(H_{1})$ be a
completely positive map and
$\Phi:\mathcal{E}\rightarrow\mathcal{B}(H_{1},H_{2})$ be a linear
map.
 Then $\Phi$ is a completely semi-$\varphi$-map if and only if
$\begin{bmatrix}
id & \Phi \\
\Phi^* & \varphi
\end{bmatrix}:S_\mathcal{A}(\mathcal{E})\rightarrow\mathcal{B}(H_2\oplus H_1)$
 is a completely positive map
\end{lemma}
\begin{proof}
Let $\Phi$ be a completely semi-$\varphi$-map and $\begin{bmatrix}
1 & x \\
x^* & a
\end{bmatrix} $ a positive element of $S_{\mathcal{A}}(\mathcal{E})$.
By above lemma, $\langle x, x\rangle_{\mathcal{A}}  \leq a$ and so
$\Phi(x)^*\Phi(x) \leq \varphi(\langle x, x\rangle_{\mathcal A})
\leq \varphi(a)$.
 Then, $\begin{bmatrix}
1 & \Phi(x) \\
\Phi(x)^* & \varphi(a)
\end{bmatrix} $ is a positive element of $\mathcal{B}(H_2 \oplus H_1)$
and hence  $\begin{bmatrix}
id & \Phi \\
\Phi^* & \varphi
\end{bmatrix}$ is a positive mapping.
 To show that $\begin{bmatrix}
id & \Phi \\
\Phi^* & \varphi
\end{bmatrix}$ is a completely positive map, let $(\begin{bmatrix}
\lambda_{i,j} & T_{i,j} \\
S_{i,j}^* & a_{i,j}
\end{bmatrix})_{i,j=1}^n$ be a positive element of $\mathbb{M}_{n}(S_{\mathcal{A}}(\mathcal{E}))$ for some $n \in \mathbb{N}$.
By a unitary equivalence we have
$$(\begin{bmatrix}
\lambda_{i,j} & T_{i,j} \\
S_{i,j}^* & a_{i,j}
\end{bmatrix})_{i,j=1}^n
\cong\begin{bmatrix}
(\lambda_{i,j})_{i,j}^n & (T_{i,j})_{i,j}^n \\
(S_{i,j}^*)_{i,j}^n & (a_{i,j})_{i,j}^n
\end{bmatrix}\in\mathbb{M}_{2}(\mathcal{L}_{1}(\mathbb{M}_{n}(\mathcal{E})))\hspace{1cm}(1).$$
Then $(\lambda_{i,j})_{i,j}^n$ and $(a_{i,j})_{i,j}^n$ are
positive matrices.
 First, we assume that
$\lambda:=(\lambda_{i,j})_{i,j}^n \in \mathbb{M}_n(\mathbb{C})$ is
an invertible matrix.
 Set $T=(T_{i,j})_{i,j}^n \in
\mathbb{M}_n(\mathcal{E})$ and $a=(a_{i,j})_{i,j}^n
\in\mathbb{M}_{n}(\mathcal A)$, then
$$\begin{bmatrix}
I_n & \lambda^{-\frac{1}{2}}T \\
 T^*\lambda^{-\frac{1}{2}}  & a
\end{bmatrix}
=
\begin{bmatrix}
 \lambda^{-\frac{1}{2}} &0 \\
0 & I_n
\end{bmatrix}
\begin{bmatrix}
\lambda & T\\
T^*& a
\end{bmatrix}
\begin{bmatrix}
 \lambda^{-\frac{1}{2}} &0 \\
0 & I_n
\end{bmatrix}$$
 is a positive element of $\mathbb{M}_{2}(\mathcal{L}_{1}(\mathbb{M}_{n}(\mathcal{E})))$.
Thus $\langle \lambda^{-\frac{1}{2}}T,\lambda^{{-\frac{1}{2}}}T
\rangle_{\mathbb{M}_n(\mathcal A)} \leq a$
 and hence
$$\Phi_n(\lambda^{-\frac{1}{2}}T)^*\Phi_n(\lambda^{-\frac{1}{2}}T)
\leq \varphi_{n}(\langle \lambda^{-\frac{1}{2}}T, \lambda^{-\frac{1}{2}} T
\rangle_{\mathbb{M}_n(\mathcal A)})
 \leq \varphi_n(a).$$
Then
$$\begin{bmatrix}
I_n & \lambda^{-\frac{1}{2}} \Phi_n(T)\\
\Phi_n( T)^*\lambda^{-\frac{1}{2}}& \varphi_n(a)
\end{bmatrix}
=\begin{bmatrix}
I_n & \Phi_n(\lambda^{-\frac{1}{2}}T)\\
\Phi_n(\lambda^{-\frac{1}{2}}T)^*& \varphi_n(a)
\end{bmatrix}$$
is positive.
 Therefore
$$(\begin{bmatrix}
\lambda_{i,j} & \Phi(T_{i,j}) \\
\Phi(T_{j,i})^* & \varphi(a_{i,j})
\end{bmatrix})_{i,j}^n
\cong
\begin{bmatrix}
\lambda &  \Phi_n(T)\\
\Phi_n( T)^* & \varphi_n(a)
\end{bmatrix}=
\begin{bmatrix}
 \lambda^{\frac{1}{2}} &0 \\
0 & I_n
\end{bmatrix}
\begin{bmatrix}
I_n & \lambda^{-\frac{1}{2}} \Phi_n(T)\\
\Phi_n( T)^*\lambda^{-\frac{1}{2}}& \varphi_n(a)
\end{bmatrix}
\begin{bmatrix}
 \lambda^{\frac{1}{2}} &0 \\
0 & I_n
\end{bmatrix}$$ is positive too.
This means that $\begin{bmatrix}
id & \Phi \\
\Phi^* & \varphi
\end{bmatrix}$ is a completely positive map.
In general case, if $\lambda$ is not invertible we can use $\lambda + r I_n$ for some $r >0$.

Conversely, assume $\begin{bmatrix}
id & \Phi \\
\Phi^* & \varphi
\end{bmatrix}$ is a completely positive map on $S_{\mathcal A}(\mathcal{E}).$
Since for every $x\in \mathbb{M}_{n}(\mathcal{E})$,
$\begin{bmatrix}
1 & x \\
 x^* & \langle x,x\rangle
\end{bmatrix}$ is a positive element of $S_{\mathbb{M}_n(\mathcal A)}(\mathbb{M}_n
(\mathcal{E}))$,
  $\begin{bmatrix}
1 & \Phi_{n}(x) \\
\Phi_{n}(x)^* & \varphi_{n}(\langle x,x\rangle)
\end{bmatrix}$ is positive. Therefore $\Phi$ is a completely semi-$\varphi$-map.
\end{proof}

\begin{theorem}\label{t3.2}
Let $\varphi:\mathcal{A}\rightarrow\mathcal{B}(H_{1})$ be a
completely positive map and
$\Phi:\mathcal{E}\rightarrow\mathcal{B}(H_{1},H_{2})$ a completely
semi-$\varphi$-map.
 Then there exists a unital completely positive map $\psi:\mathbb{K}_{1}(\mathcal{E})\rightarrow\mathcal{B}(H_{2})$ such that
$\begin{bmatrix}
 \psi& \Phi \\
\Phi^* & \varphi
\end{bmatrix}:\mathcal{L}_{1}(\mathcal{E}) \rightarrow
\mathcal{B}(H_{2} \oplus H_{1})$ is a completely positive map.
\end{theorem}
\begin{proof}
By the above lemma, $\theta_0=\begin{bmatrix}
id & \Phi \\
\Phi^* & \phi
\end{bmatrix}:S_{\mathcal A}(\mathcal{E})\rightarrow\mathcal{B}(H_2 \oplus H_1)$ is
a completely positive map. Without loss of generality we can
assume that $\theta_0$ is unital, since for every
positive real number $r$, the map $\begin{bmatrix}
id & r\Phi \\
r\Phi^* & r^2\varphi
\end{bmatrix}$ is completely positive.
By Arveson's extension theorem, $\theta_0$ has a unital  completely positive extension
$\theta:\mathcal{L}_{1}(\mathcal{E})\rightarrow\mathcal{B}(H_2\oplus
H_1)$. Put $p:= \begin{bmatrix}
I & 0 \\
0 & 0 \end{bmatrix}.$
By \cite[Corollary 5.2.2]{Effros}, for each $u\in\mathbb{K}_1(\mathcal{E})$ we have
$$\theta(\begin{bmatrix}
u & 0 \\
0 & 0 \end{bmatrix})=
\theta(p\begin{bmatrix}
u & 0 \\
0 & 0 \end{bmatrix}p)
=\theta(p)\theta(\begin{bmatrix}
u & 0 \\
0 & 0 \end{bmatrix})\theta(p)
=\begin{bmatrix}
id_{H_2} & 0 \\
0 & 0 \end{bmatrix}\theta(\begin{bmatrix}
u & 0 \\
0 & 0 \end{bmatrix})\begin{bmatrix}
id_{H_2} & 0 \\
0 & 0 \end{bmatrix}
 \in \begin{bmatrix}
\mathcal{B}(H_2) & 0 \\
0 & 0 \end{bmatrix}$$

Thus, $\theta$ is a corner preserving unital completely positive map on $\mathcal{L}_1(\mathcal{E}).$
Therefore, $\psi:=\theta|_{\mathbb{K}_{1}(\mathcal{E})}$ is a
unital completely positive map from $\mathbb{K}_{1}(\mathcal{E})$
into $\mathcal{B}(H_{2})$ such that $\theta=\begin{bmatrix}
\psi &   \Phi \\
 \Phi^* & \varphi
\end{bmatrix}$.
\end{proof}

The next theorem is a strengthened form of the main theorem of
\cite{Asadi}.

\begin{theorem}\label{t3.3}
Let $\varphi:\mathcal{A}\rightarrow\mathcal{B}(H_{1})$ be a
completely positive map and
$\Phi:\mathcal{E}\rightarrow\mathcal{B}(H_{1},H_{2})$ be a
completely semi-$\varphi$-map. Then there exist Hilbert spaces
 $K_{1},K_{2}$,  a bounded operator
 \linebreak$V:H_{1}\rightarrow K_{1}$, an isometry $W:{H}_{2}\rightarrow K_{2}$, a $*$-homomorphism
$\rho:\mathcal{A}\rightarrow\mathcal{B}(K_{1})$ and a
$\rho$-representation
    $\Psi:\mathcal{E}\rightarrow\mathcal{B}(K_{1},K_{2})$ such
    that for all $a\in \mathcal A$ and $x\in\mathcal{E},$
$$\hspace{1cm}
\varphi(a)=V^*\rho(a)V\hspace{1cm}\Phi(x)=W^*\Psi(x)V.$$
Furthermore, if $\varphi$ is unital, then $V$ is an isometry.
\end{theorem}
\begin{proof}
By the previous theorem,
$\begin{bmatrix}
id & \Phi \\
\Phi^* & \varphi
\end{bmatrix}:S_{\mathcal{A}}(\mathcal{E})\rightarrow\mathcal{B}(H_{2}\oplus H_{1})$
has a completely positive map extension
 $\theta=\begin{bmatrix}
\psi & \Phi \\
\Phi^* & \varphi
\end{bmatrix}:\mathcal{L}_{1}(\mathcal{E})\rightarrow\mathcal{B}(H_{2}\oplus H_{1})$.

By Stinespring's theorem for a completely positive maps on
$C^*$-algebras
 there is a triple $(\pi,K,W)$ consists of a unital $*$-representation
 \hbox{$\pi:\mathcal{L}_{1}(\mathcal{E})\rightarrow\mathcal{B}(K)$}
 and an operator
  $W\in\mathcal{B}(H_{2}\oplus H_{1},K)$ such that
for every $X\in\mathcal{L}_{1}(\mathcal{E})$ the following holds
$$\theta(X)=W^*\pi(X)W .$$

Similar to \cite[Proposition 3.1]{Arambašic}, we set
$K_{1}=[\pi(\begin{bmatrix}
0 & 0 \\
0 & \mathcal{A}
\end{bmatrix})K]$ and
$K_{2}=[\pi(\begin{bmatrix}
\mathbb{K}_{1}(\mathcal{E}) & 0 \\
0 & 0
\end{bmatrix})K].$
Hence   $K \cong K_2 \oplus K_1$ and we can write
$\pi=\begin{bmatrix}
\sigma & \Psi \\
\Psi^* & \rho
\end{bmatrix} $
and $W=\begin{bmatrix}
W_1 & W_2 \\
W_3 & W_4
\end{bmatrix} \in\mathcal{B}(H_{2}\oplus H_{1},K_{2}\oplus K_{1}),$
where $\sigma:\mathbb{K}_1(\mathcal{E})\rightarrow\mathcal{B}(K_2)$
and $\rho: \mathcal{A}\rightarrow\mathcal{B}(K_1)$  are
$*$-representations and
$\Psi:\mathcal{E}\rightarrow\mathcal{B}(K_1,K_2)$ is a
$\sigma$-$\rho$-representation.

 Then for every $\begin{bmatrix}
T & x \\
y^* & a
\end{bmatrix} \in\mathcal{L}_{1}(\mathcal{E})$ we have
$$\begin{bmatrix}
\psi(T) & \Phi(x) \\
\Phi(y)^* & \varphi(a)
\end{bmatrix}=\begin{bmatrix}
W_{1}^* & W_{3}^* \\
W_{2}^* & W_{4}^*
\end{bmatrix}\begin{bmatrix}
\sigma(T) & \Psi(x) \\
\Psi(y)^* & \rho(a)
\end{bmatrix} \begin{bmatrix}
W_{1} & W_{2} \\
W_{3} & W_{4}
\end{bmatrix}\hspace{2cm}(2).$$
In the above equation, set $T=I_{\mathcal{E}},$ $x=y=0$ and $a=0.$
Since $\sigma$ and $\rho$ are unital maps, one has
$$\begin{bmatrix}
id_{H_{2}} & 0 \\
0 & 0
\end{bmatrix}=\begin{bmatrix}
W_{1}^* & W_{3}^* \\
W_{2}^* & W_{4}^*
\end{bmatrix}\begin{bmatrix}
id_{K_{2}} & 0 \\
0 & 0
\end{bmatrix} \begin{bmatrix}
W_{1} & W_{2} \\
W_{3} & W_{4}
\end{bmatrix},$$
thus $W_{1}^*W_{1}=id_{H_{2}}$ and $W_{2}^*W_{2}=0.$ Now, set
$T=0,$ $y=0,$ $a=0$ and an arbitrary $x\in\mathcal{E}$ in equation
$(2),$ easy calculation shows that $\Phi(x)=W_{1}^*\Psi(x)W_{4}.$
Finally, setting $T=0,$ $x=y=0$ and an arbitrary element $a\in A$
in equation $(2)$ shows that $\varphi(a)=W_{4}^*\rho(a)W_{4}$ and
$W_{3}^*\rho(a)W_{3}=0.$
 Since $\rho$ is unital, $W_{3}=0.$
  If $\varphi$ is unital, one has
  $$id_{H_{1}}=\varphi(1)=W_{4}^*\rho(1)W_{4}=W_{4}^*id_{K_{1}}W_{4}=W_{4}^*W_{4}.$$
 \end{proof}

 We summarize the results of this section on completely semi-$\varphi$-maps in the following corollary:

\begin{corollary}\label{c3.4}
Let $\mathcal A$ be a unital $C^*$-algebra and $\mathcal{E}$ a
right Hilbert $\mathcal A$-module.
 For every pair of given maps $\varphi:\mathcal{A}\rightarrow\mathcal{B}(H_{1})$ and
 $\Phi:\mathcal{E}\rightarrow\mathcal{B}(H_{1},H_{2})$ the following are equivalent:
\item(i) $\Phi$ is a completely semi-$\varphi$-map
\item(ii) $\begin{bmatrix}
id & \Phi \\
\Phi^* & \varphi
\end{bmatrix}:S_{\mathcal{A}}(\mathcal{E})\rightarrow\mathcal{B}(H_2
\oplus H_1)$ is a completely positive map
\item(iii) There exists a unital completely positive map $\psi:\mathbb{K}_{1}(\mathcal{E})\rightarrow\mathcal{B}(H_{2})$
such that $\begin{bmatrix}
\psi & \Phi \\
\Phi^* & \varphi
\end{bmatrix}:\mathcal{L}_{1}(\mathcal{E})\rightarrow\mathcal{B}(H_{2}\oplus H_{1})$
is a completely positive map
\item(v) There exist Hilbert spaces $K_{1},K_{2}$, a bounded
operator $V:H_{1}\rightarrow K_{1}$, an isometry
$\hspace{1mm}W:H_{2}\rightarrow K_{2}$, and a unital
$*$-representation $\pi: \mathcal{L}_{1}(\mathcal{E}) \rightarrow
\mathcal{B}(K_{2} \oplus K_{1})$ such that $$
\begin{bmatrix}
* & \Phi\\
\Phi^*  & \varphi
\end{bmatrix}(\cdot)=\begin{bmatrix}
W^* & 0 \\
0 & V^*
\end{bmatrix} \pi(\cdot) \begin{bmatrix}
W & 0 \\
0 & V
\end{bmatrix},$$
\item(iv) There exists a pair $((\rho,K_{1},V),(\Psi,K_{2},W))$ consists of
Hilbert spaces $K_{1},K_{2}$, a bounded operator
$V:H_{1}\rightarrow K_{1}$, an isometry
$\hspace{1mm}W:H_{2}\rightarrow K_{2},$ a unital
$*$-representation
$\rho:\mathcal{A}\rightarrow\mathcal{B}(K_{1})$,
 and a $\rho$-representation $\Psi:\mathcal{E}\rightarrow\mathcal{B}(K_{1},K_{2})$ such that
$$\varphi(a)=V^*\rho(a)V,\hspace{1cm}\Phi(x)=W^*\Psi(x)V,$$
for all $a\in A$ and $x\in\mathcal{E}.$
\end{corollary}
\begin{proof}
$(i)\Leftrightarrow(ii)$ by Lemma \ref{l3.1}.
 $(v)\Leftrightarrow(iv)$ and $(v)\Rightarrow(iii)$ by Remark \ref{r2.1}.
  $(i)\Rightarrow(iv)$ by Theorem \ref{t3.3} and obviously $(iii)\Rightarrow(ii)$.
\end{proof}


\section{UNIQUENESS OF MINIMAL DILATION PAIRS}

 Assume $\mathcal{A}$ is a unital $C^*$-algebra and $\mathcal{E}$ is a right Hilbert $\mathcal{A}$-module.
As it is shown in the previous section, if
$\varphi:\mathcal{A}\to\mathcal{B}(H)$ is a completely positive
map, every completely semi-$\varphi$-map on $\mathcal{E}$ is
dilatable. In this section, we show that every completely
semi-$\varphi$-map on a Hilbert $C^*$-module  has a minimal
dilation pair. Furthermore, we show that two minimal dilation
pairs for a given completely semi-$\varphi$-map are unitarily
equivalent and implement unitarily equivalent $*$-representations
on the linking $C^*$-algebra of $\mathcal{E}$.

\begin{definition}\label{d4.1}
Let $\mathcal{A}$ be a $C^*$-algebra  and $\mathcal{E}$ a right Hilbert $\mathcal A$-module.
 A map  $\Phi:\mathcal{E}\to\mathcal{B}(H,K)$ is a \textit{CP-extendable} map,
if there exist completely positive maps
$\psi:\mathbb{K}_{1}(\mathcal{E})\rightarrow\mathcal{B}(K)$ and
$\varphi:\mathcal{A}\to\mathcal{B(H)}$ such that $\begin{bmatrix}
\psi & \Phi \\
\Phi^* & \varphi
\end{bmatrix}: \mathcal{L}_{1}(\mathcal{E}) \rightarrow \mathcal{B}(K \oplus
H)$ is a completely positive map. In this case we call
 the pair $(\varphi,\Phi)$ a \textit{CP-extendable pair}.
\end{definition}

As it is shown in Corollary \ref{c3.4}, if
$\varphi:\mathcal{A}\rightarrow\mathcal{B}(H)$ is a completely
positive map, then,  every completely semi-$\varphi$-map
  $\Phi:\mathcal{E}\rightarrow\mathcal{B}(H,K)$ has a CP-extension on the linking $C^*$-algebra $\mathcal{L(E)}$ which acts block-wise.
   Thus $\Phi$ is a CP-extendable map and $(\varphi,\Phi)$ is a CP-extendable pair.

 Since completely semi-$\varphi$-maps are CP-extendable, the following theorem is a generalization of Corollary \ref{c3.4} and can
be proved, by using  Stinespring's theorem for linking
$C^*$-algebra and \cite[Proposition 3.1]{Arambašic}.

\begin{theorem}\label{t4.2}
Let $\mathcal{A}$ be a unital $C^*$-algebra and $\mathcal{E}$ a right Hilbert $\mathcal{A}$-module.
 For a given map
 $\Phi:\mathcal{E}\to\mathcal{B}(H_{1},H_{2})$ the following are equivalent:
\item(i) $\Phi$ is CP-extendable,
\item(ii) There exist Hilbert spaces $K_{1},K_{2}$, bounded
operators $V:H_{1}\to K_{1},\hspace{1mm}W:H_{2}\to K_{2}$ and a
unital $*$-representation $\pi: \mathcal{L}_{1}(\mathcal{E}) \to
\mathcal{B}(K_{2} \oplus K_{1})$ and a completely positive map
$\varphi:\mathcal{A}\to\mathcal{B}(H_1)$ such that
$$
\begin{bmatrix}
* & \Phi\\
\Phi^*  & \varphi
\end{bmatrix}(\cdot)=\begin{bmatrix}
W^* & 0 \\
0 & V^*
\end{bmatrix} \pi(\cdot) \begin{bmatrix}
W & 0 \\
0 & V
\end{bmatrix},$$
\item(iii) $\Phi$ is \textit{dilatable}.
\end{theorem}

In the following we recall a definition from \cite{BRS} and show
that \cite[Theorem 2.4]{BRS} holds
 for completely semi-$\varphi$-maps.

\begin{definition}\label{d4.3}
Let $\varphi:\mathcal{A}\rightarrow\mathcal{B}(H_1)$ be a
completely positive map and
$\Phi:\mathcal{E}\rightarrow\mathcal{B}(H_{1},H_{2})$ be a
completely semi-$\varphi$-map.
 A \textit{dilation pair} for $(\varphi,\Phi)$ is a pair of triples $((\rho,K_{1},V),(\Psi,K_{2},W))$ consists of
Hilbert spaces $K_{1},K_{2}$,
a unital $*$-representation
$\rho:\mathcal{A}\rightarrow\mathcal{B}(K_{1})$
 and a $\rho$-morphism $\Psi:\mathcal{E}\rightarrow\mathcal{B}(K_{1},K_{2})$ and bounded
operators
$V:H_{1}\rightarrow K_{1}$ and $W:H_{2}\rightarrow K_{2},$ such that
$$\Phi(x)=W^*\Psi(x)V \hspace{5mm} ,\hspace{5mm}  \varphi(a)=V^*\rho(a)V,$$
for all $a\in \mathcal{A}$ and $x\in\mathcal{E}.$
   A dilation pair $((\rho,K_{1},V),(\Psi,K_{2},W))$ is called \textit{minimal} when the following conditions are satisfied
\item (i) $[\rho(\mathcal{A})V H_1]=K_1$,
\item (ii) $\Psi$ be a nondegenerate map.

\end{definition}

Suppose that $\varphi:\mathcal{A}\rightarrow\mathcal{B}(H_1)$ is a
completely positive map and
 $\Phi:\mathcal{E}\rightarrow\mathcal{B}(H_1,H_2)$ is a completely semi-$\varphi$-map.
 By Theorem \ref{t3.3} there exists a dilation pair $((\rho,K_{1},V),(\Psi,K_{2},W))$ for $(\varphi,\Phi),$ such that $W$ is an isometry.
   We can replace $(\rho,K_{1},V)$
    by a minimal Stinespring dilation triple for $\varphi$.
    So without loss of generality we can assume that $(\rho,K_{1},V)$ is a minimal Stinespring dilation triple for $\varphi.$
     Set
         $\mathcal{L}:=[\Psi(\mathcal{E})VH_{1}]=[\Psi(\mathcal{E})K_1]$ and define
          $\Gamma:\mathcal{E}\rightarrow\mathcal{B}(K_{1},L)$ by
$$\Gamma(x)k:=\Psi(x)k$$
for all $x\in\mathcal{E}$ and $k\in K_{1}.$

 Assume $\jmath:L\rightarrow K_{2}$ is the inclusion map of
 $L$ into $K_{2}$, so, $\jmath^*$ is the orthogonal projection from
 $K_{2}$ onto $L.$ Thus $\Gamma(x)=\jmath^*\Psi(x)$ for every $x\in\mathcal{E}.$
  Note that $\Psi$ is a $\rho$-morphism, therefore $\Psi$ is a $\rho$-module map.
   Thus for every $a\in \mathcal A,$ $h\in H_{1}$ and $x,y\in\mathcal{E}$
$$\Gamma(x)^*\Gamma(y)(\rho(a)Vh)=\rho(\langle x,y\rangle)(\rho(a)Vh).$$
Since $[\rho(A)VH_{1}]=K_{1},$ $\Gamma$ is a $\rho$-map.
 Now define
$T:L\rightarrow H_{2}$ by $T(l):=W^*(l)$
for all $l\in L.$
 Consider
$S:=T^*\in\mathcal{B}(H_{2},L),$ then
$((\rho,K_{1},V),(\Gamma,L,S))$ is a minimal dilation pair for
$(\varphi,\Phi).$ Note that $W$ is an isometry and $T=W^*\jmath,$
thus $T$ and $S=T^*$ are contractions with norm one.

The following theorem on the uniqueness of minimal dilation pairs of completely semi-$\varphi$-maps
 is in fact the same as \cite[Theorem 2.4]{BRS}.

\begin{theorem}\label{t4.4}
Let $\Phi$ and $\varphi$ be as in definition \ref{d4.3}. Assume
$((\rho,K_{1},V),(\Psi,K_{2},W))$  and
$((\pi,L_{1},U),(\Gamma,L_{2},S))$ are two minimal dilation pairs
for $(\varphi,\Phi).$ Then there exist unitary operators
$T_{1}:K_{1}\rightarrow L_{1}$ and $T_{2}:K_{2}\rightarrow L_{2}$
such that
\item(i) $T_{1}V=U$ and $T_{1}\rho(a)=\pi(a)T_{1}$ for all $a\in \mathcal A$.
\item(ii) $T_{2}W=S$ and $T_{2}\Psi(x)=\Gamma(x)T_{1}$ for all $x\in\mathcal{E}$.
\item(iii) $\begin{bmatrix}
T_{2} & 0\\
0  & T_{1}
\end{bmatrix}\begin{bmatrix}
W & 0\\
0  & V
\end{bmatrix}=\begin{bmatrix}
S & 0\\
0  & U
\end{bmatrix}$ and $\begin{bmatrix}
T_{2} & 0\\
0  & T_{1}
\end{bmatrix}\begin{bmatrix}
\sigma & \Psi\\
\Psi^*  & \rho
\end{bmatrix}\begin{bmatrix}
T_{2}^* & 0\\
0  & T_{1}^*
\end{bmatrix}=\begin{bmatrix}
\tau & \Gamma\\
\Gamma^*  & \pi
\end{bmatrix}$, where $\sigma:\mathbb{K}_1(\mathcal{E})\to\mathcal{B}(K_2)$
and $\tau:\mathbb{K}_1(\mathcal{E})\to\mathcal{B}(L_2)$
are unique $*$-homomorphisms which satisfy the equations
$\sigma(x\otimes y)=\Psi(x)\Psi(y)^*$ and $\tau(x\otimes
y)=\Gamma(x)\Gamma(y)^*$, for all $x,y\in\mathcal{E}$.

Consequently, representations $\rho$ and $\begin{bmatrix}
\sigma & \Psi\\
\Psi^*  & \rho
\end{bmatrix}$ and $\sigma$ are unitarily equivalent
to representations $\pi$ and $\begin{bmatrix}
\tau & \Gamma\\
\Gamma^*  & \pi
\end{bmatrix}$ and $\tau$, respectively.
\end{theorem}
\begin{proof}
$(i)$ and $(ii)$ have the same proof as \cite[Theorem 2.4]{BRS} and
$(iii)$ can be obtained from $(i)$ and $(ii)$.
\end{proof}

\begin{remark}\label{r4.5}
Let $\varphi:\mathcal{A}\rightarrow\mathcal{B}(H_1)$ be a
completely positive map and
$\Phi:\mathcal{E}\rightarrow\mathcal{B}(H_1,H_2)$ be a completely
semi-$\varphi$-map.
 By the preceding discussion on the existence of a minimal dilation pair for completely semi-$\varphi$-maps, there is a
 minimal dilation pair $((\pi,L_{1},U),(\Gamma,L_2,S))$ for $(\varphi,\Phi)$ such that $S$ is contractive.
  Theorem \ref{t4.4} implies that for every minimal  dilation pair
   $((\pi',L'_{1},U'),(\Gamma',L'_2,S'))$ for $(\varphi,\Phi),$ $S'$ is contractive.
  On the other hand by  \cite[Theorem 2.1]{BRS} for every $\varphi$-map
  there is a minimal dilation pair  $((\rho,K_{1},V),(\Psi,K_2,W))$  such that $W$ is coisometry  , thus, by Theorem
  \ref{t4.4} if $((\rho',K'_{1},V'),(\Psi',K'_2,W'))$ is an another minimal
  dilation pair for the $\varphi$-map, then $W'$ is coisometry.
   Therefore there exist
   many examples of completely semi-$\varphi$-maps which are not $\varphi$-map.

\end{remark}

In the following we show that this new notion of dilation for
completely semi-$\varphi$-maps is compatible with the previous
notion of minimal dilation pair for completely positive maps on
$C^*$-algebras. For this purpose we recall the definition of
irreducible maps on Hilbert $C^*$-modules and show that a unital
completely positive map on a $C^*$-algebra is pure if and only if
its minimal dilation pair (in sense of Definition \ref{d4.3}) is
irreducible.


 \begin{definition}\label{d4.6} Let
$\Psi:\mathcal{E}\rightarrow\mathcal{B}(H_1.H_2)$ be a map and
$K_1\leq H_1$ and $K_2\leq H_2.$ The pair $(K_1,K_2)$ is said to be
$\Psi$-invariant if $\Psi(\mathcal{E})K_1\subseteq K_2$ and
$\Psi(\mathcal{E})^* K_2\subseteq K_1.$ $\Psi$ is said to be
irreducible if $(0,0)$ and $(H_1,H_2)$ are the only $\Psi$-invariant
pairs.
 \end{definition}


 \begin{remark} \label{r4.7} The above definition is a
modification of \textbf{Definition 3.3 \cite{Arambašic}}, just we
state it for every map not just representations.
Aramba$\check{s}$i$\acute{c}$ showed that if
$\rho:\mathcal{A}\rightarrow\mathcal{B}(H)$ is a $*$-representation
and $\Psi:\mathcal{E}\rightarrow \mathcal{B}(\mathcal{H,K})$ is a
$\rho$-representation for $\mathcal{E}$ such that
$[\Psi(\mathcal{E})H]=K$ then $\Psi$ is irreducible (in sense of
Definition \ref{d4.6} )
  if and only if $\rho$ is irreducible \cite[Proposition 3.6]{Arambašic}.
\end{remark}

 By a result of Arveson \cite[Corollary
1.4.3]{Arv} the completely positive map $\varphi$ is pure if and
only if it can be dilated to an irreducible $*$-representation of
$\mathcal{A}$ ( in other words, its minimal Stinespring dilation
triple is irreducible ). The following corollary is a generalization
of this fact.


\begin{corollary}
Let $\mathcal{A}$ be a unital $C^*$-algebra
and $\varphi:\mathcal{A}\rightarrow\mathcal{B}(H)$ be a unital
completely positive map.
 Assume $((\rho,K_1,V),(\Psi,K_2,W))$ is the minimal dilation pair for $(\varphi,\varphi).$
  Then $\varphi$ is pure if and only if $\Psi$ is an irreducible map.
\end{corollary}
\begin{proof}
Assume  $\varphi$ is pure, thus by  \cite[Corollary 1.4.3]{Arv}
every minimal Stinespring dilation triple of it is irreducible.
Since $((\rho,K_1,V),(\Psi,K_2,W))$ is a minimal dilation pair for
$(\varphi,\varphi),$ $\rho$ is an irreducible $*$-representation and
$[\Psi(\mathcal{E})K_1]=K_2,$ therefore $\Psi$ is an irreducible
representation of $\mathcal{E}$ by \cite[Proposition
3.6]{Arambašic}. Conversely, If $((\rho,K_1,V),(\Psi,K_2,W))$ is a
dilation pair for $(\varphi,\varphi)$ such that $\Psi$ is
irreducible, then \cite[Lemma 3.5]{Arambašic} implies that $\rho$ is
an irreducible $*$-representation for $\mathcal{A}.$ Thus
$(\rho,K_1,V)$ is a minimal Stinespring dilation triple for
$\varphi,$ and $((\rho,K_1,V),(\Psi,K_2,W))$ is a minimal dilation
pair for $(\varphi,\varphi).$
  Therefore $\varphi$ is pure by  \cite[Corollary 1.4.3]{Arv}.
\end{proof}

\section{A RADON-NIKODYM-TYPE THEOREM FOR COMPLETELY SEMI-$\varphi$-MAPS}

We denote the set of all pairs $(\varphi,\Phi)$, where
$\varphi:\mathcal{A}\rightarrow\mathcal{B}(H_1)$ is a completely
positive map and   $\Phi:\mathcal{F}\to\mathcal{B}(H_{1}, H_{2})$
is a completely semi-$\varphi$-map, by $\mathcal{CPE(F},H_{1},
H_{2})$.

Let $(\varphi,\Phi)\in\mathcal{CPE}(\mathcal{F}, H_{1}, H_{2}).$
Assume $((\rho, K_{1}, V),(\Psi, K_{2}, W))$ is a minimal
 dilation pair for $(\varphi,\Phi).$
Then $\rho$ is unital and by Remark \ref{r2.1} there
exists a $*$-homomorphism
 $\sigma:\mathbb{K}_{1}(\mathcal{F})\rightarrow\mathcal{B}(K_{2})$
  such that $\sigma(x\otimes y)=\Psi(x)\Psi(y)^*$ and moreover
  $$\begin{bmatrix}
\sigma & \Psi\\
\Psi^*  & \rho
\end{bmatrix}:\mathcal{L}_{1}(\mathcal{F})\rightarrow\mathcal{B}(K_{2}\oplus K_{1})
, \ \
\begin{bmatrix}
u & x\\
y^*  & a
\end{bmatrix}\mapsto \begin{bmatrix}
\sigma(u) & \Psi(x)\\
\Psi(y)^*  & \rho(a)
\end{bmatrix}.$$
is a $*$-representation.

The range of $\begin{bmatrix}
\sigma & \Psi\\
\Psi^*  & \rho
\end{bmatrix}$ is a $C^*$-subalgebra of $\mathcal{B}(K_{2}\oplus K_{1})$
 and it is easy to check that its commutant is the set of all $\begin{bmatrix}
P & 0\\
0  & Q
\end{bmatrix}\in\mathcal{B}(K_{2}\oplus K_{1})$ such that
\begin{equation}P\Psi(x)=\Psi(x)Q \hspace{1cm},\hspace{1cm}Q\Psi(x)^*=\Psi(x)^*P,\end{equation}
\begin{equation}\sigma(u)P=P\sigma(u)\hspace{15mm},\hspace{1cm}\rho(a)Q=Q\rho(a)\end{equation}
 for all $x\in\mathcal{F},$ $a\in \mathcal{A}$ and $u \in\mathbb{K}_{1}(\mathcal{F}).$

If $\mathcal{F}$ is full, then (1) implies (2).
  The above discussion lead us to
  the following definition \cite[Definition 4.1]{Arambašic}.


\begin{definition}\label{d5.1}
Let $\rho:\mathcal{A}\rightarrow B(H_{1})$ be a unital
$*$-representation
 and $\Psi:\mathcal{F}\rightarrow\mathcal{B}(H_{1}, H_{2})$ a $\rho$-map.
Commutant of $\Psi$ is the set of all operators $\begin{bmatrix}
P & 0\\
0  & Q
\end{bmatrix}\in\mathcal{B}(H_{2}\oplus H_{1})$
 such that the following equations hold for all $x\in\mathcal{F}$
$$P\Psi(x)=\Psi(x)Q \hspace{1cm},\hspace{1cm}Q\Psi(x)^*=\Psi(x)^*P,$$
and is denoted by $\Psi(\mathcal{F})'.$
\end{definition}

\begin{remark}\label{remark5.2}
 Assume $\Psi$ and $\rho$ as in the Definition \ref{d5.1}. Then $\Psi(\mathcal{F})'$ is a $C^*$-algebra,
moreover $(\begin{bmatrix}
\sigma & \Psi \\
\Psi^* & \rho
\end{bmatrix}(\mathcal{L}(\mathcal{F})))'\subseteq\Psi(\mathcal{F})'$.
In the case of full Hilbert $C^*$-modules if $\Psi$ is non-degenerate
 $([\Psi(\mathcal{F})H_1]=H_2),$ then  $[\Psi(\mathcal{F})^*H_2]=H_1$ and
$(\begin{bmatrix}
\sigma & \Psi \\
\Psi^* & \rho
\end{bmatrix}(\mathcal{L}(\mathcal{F})))'=\Psi(\mathcal{E})'$ \cite[Lemma 4.3 and Lemma 4.4]{Arambašic}.
\end{remark}

From now on we deal with full Hilbert $C^*$-modules.
 In the following we define an order relation on $\mathcal{CPE(F},H,K)$
and prove a Radon-Nikodym type theorem for this class of maps.

\begin{definition}\label{d5.2}
Let $(\varphi_i,\Phi_i)\in\mathcal{CPE(F},H,K)$ for $i=1,2.$
 We say that $(\varphi_1,\Phi_1)\ll(\varphi_2,\Phi_2)$ when
$$\begin{bmatrix}
id & \Phi_1\\
\Phi_1^*  & \varphi_1
\end{bmatrix}\leq_{cp}\begin{bmatrix}
id & \Phi_2\\
\Phi_2^*  & \varphi_2
\end{bmatrix},$$
where $\leq_{cp}$ is the order on the set of completely positive maps from
 $S_\mathcal{A}(\mathcal{F})$ into $\mathcal{B}(K\oplus H).$
\end{definition}

We use the notation $T\oplus S$ instead of $\begin{bmatrix}
T & 0\\
0  & S
\end{bmatrix}$ for operators $T\in\mathcal{B}(H)$ and $S\in\mathcal{B}(K).$
Note that $T\oplus S$ is a positive operator on
$H\oplus K$ if and only if $T\in\mathcal{B}(H)_{+}$
 and $S\in\mathcal{B}(K)_{+}.$ The following proposition is similar to
  \cite[Lemma 2.10]{J} on $\varphi$-maps,
  and we show that the lemma is true for completely semi-$\varphi$-maps, too.

 \begin{proposition}\label{p5.2}
 Assume $(\varphi,\Phi)\in\mathcal{CPE}(\mathcal{E},H_{1},H_{2})$ and
   $((\rho,K_{1},V),(\Psi,K_{2},W))$
   is a minimal dilation pair of $(\varphi,\Phi).$ For every  positive operator
    $T\oplus S\in\Psi(\mathcal{E})'$
     define the map
     $\Phi_{T\oplus S}:\mathcal{E}\rightarrow\mathcal{B}(H_{1},H_{2})$  by
 $$\Phi_{T\oplus S}(x):=W^*T^{\frac{1}{2}}\Psi(x)S^{\frac{1}{2}}V$$
for all $x\in\mathcal{E}.$ Then $(\varphi_S,\Phi_{T\oplus S})$
 is a CP-extendable pair, where  $\varphi_{S}(a)=V^*S\rho(a)V$, for each $a \in \mathcal{A}$.
  Moreover, if $T$ is contractive, $\Phi_{T\oplus S}$ is a completely semi-$\varphi_S$-map.
\end{proposition}
\begin{proof}
Since $((\rho,K_{1},V),(\Psi,K_{2},W))$
 is a minimal dilation for $(\varphi,\Phi),$ there exists
 a  non-degenerate (and therefore unital) $*$-homomorphism
  $\sigma:\mathbb{K}_{1}(\mathcal{E})\rightarrow\mathcal{B}(K_{2})$ such that
$$\begin{bmatrix}
\sigma & \Psi\\
\Psi^*  & \rho
\end{bmatrix}:\mathcal{L}_{1}(\mathcal{E})\rightarrow\mathcal{B}(K_{2}\oplus K_{1})$$
is a  $*$-homomorphism, so
$$\begin{bmatrix}
W^*T^{\frac{1}{2}} & 0 \\
0 & V^*S^{\frac{1}{2}}
\end{bmatrix}\begin{bmatrix}
\sigma & \Psi\\
\Psi^*  & \rho
\end{bmatrix}\begin{bmatrix}
T^{\frac{1}{2}}W & 0 \\
0 & S^{\frac{1}{2}}V
\end{bmatrix}:\mathcal{L}_{1}(\mathcal{E})\rightarrow\mathcal{B}(H_{2}\oplus H_{1})$$
is a completely positive map.
 By the discussion previous the
Definition \ref{d5.1}, $T\in\sigma(\mathbb{K}_{1}(\mathcal{E}))'$
and $S\in\rho(\mathcal{A})'$ thus
$S^{\frac{1}{2}}\in\rho(\mathcal{A})'$ and
$T^{\frac{1}{2}}\in\sigma(\mathbb{K}_{1}(\mathcal{E}))',$
 therefore
 $$\begin{bmatrix}
W^*T^{\frac{1}{2}}\sigma(.) T^{\frac{1}{2}}W & W^*T^{\frac{1}{2}}\Psi(.) S^{\frac{1}{2}}V\\
V^*S^{\frac{1}{2}}\Psi(.)^*T^{\frac{1}{2}}W  & V^*S\rho(.) V
\end{bmatrix}:\mathcal{L}_1(\mathcal{E})\rightarrow\mathcal{B}(H_{2}\oplus H_{1})$$
 is a completely positive map, thus
 $\Phi_{T\oplus S}$ is a CP-extendable map and $(\varphi_S,\Phi_{T\oplus S})$ is a CP-extendable pair.

 It is easy to check that $\Phi_{T\oplus S}$ is a completely semi-$\varphi_S$-map when $T$ is contractive.
\end{proof}

The above proposition has a converse that is
 a Radon-Nikodym type theorem for completely semi-$\varphi$-maps.

 \begin{theorem}\label{t5.4}
Let $\mathcal{E}$ be a full Hilbert $C^*$-module over a unital
$C^*$-algebra $\mathcal{A}.$ Assume \linebreak{
$(\varphi_1,\Phi_1),(\varphi_2,\Phi_2)\in\mathcal{CPE(E,}H_1,H_2)$}
and $(\varphi_1,\Phi_1)\ll(\varphi_2,\Phi_2).$
 Then, there exists a dilation pair
$((\rho,K_{1},W'),(\Psi,K_{2},W))$ for $(\varphi_2,\Phi_2)$ such that $\Psi$ is non-degenerate
 and a unique  positive contraction $T\oplus S\in\Psi(\mathcal{E})'$
  such that
  \begin{equation}\Phi_1(e)=W^*T^{\frac{1}{2}}\Psi(e)S^{\frac{1}{2}}W'\end{equation}
  for all $a\in\mathcal{A}$ and $e\in\mathcal{E}.$
\end{theorem}
\begin{proof}
Assume $(\varphi_1,\Phi_1)\ll(\varphi_2,\Phi_2),$ thus,
 $\begin{bmatrix}
id & \Phi_1 \\
\Phi_1^*  & \varphi_1
\end{bmatrix}\leq_{cp}
\begin{bmatrix}
id & \Phi_2 \\
\Phi_2^*  & \varphi_2
\end{bmatrix}$. Put $\psi_1:=\begin{bmatrix}
id & \Phi_1 \\
\Phi_1^*  & \varphi_1
\end{bmatrix}$ and $\psi_2:=\begin{bmatrix}
id & \Phi_2 \\
\Phi_2^*  & \varphi_2
\end{bmatrix}.$
Thus $\psi_2 - \psi_1$ is a completely positive map. By Arveson's
extension theorem, $ \psi_2 - \psi_1$ and $   \psi_1$ have
completely positive extensions $\widetilde{\psi_2 -\psi_1}$ and
$\widetilde{  \psi_1}$  on $\mathcal{L}_{1}(\mathcal{E})$. Let
$\widetilde{  \psi_2} :=\widetilde{  \psi_1} + \widetilde{\psi_2 -
\psi_1}$. We have $\widetilde{  \psi_1} \leq_{cb} \widetilde{
\psi_2} $.
 Thus $\widetilde{  \psi_2}$  is a completely positive
extension for $\psi_2$ such that $\widetilde{  \psi_1} \leq_{cb}
\widetilde{  \psi_2}$.
 Assume $(\pi,K,V)$ is the minimal Stinespring dilation triple for $\widetilde{\psi_{2}}$,
 then similar to the proof of Theorem \ref{t3.2}, $K$ decomposes to
  $K_{2}\oplus K_{1}$ and there exist unital $*$-homomorphisms
  $\sigma:\mathbb{K}_{1}(\mathcal{E})\rightarrow\mathcal{B}(K_{2})$
    and $\rho:\mathcal{A}\rightarrow\mathcal{B}(K_{1})$ and also a
    $\sigma$-$\rho$-representation  $\Psi:\mathcal{E}\rightarrow\mathcal{B}(K_{1},K_{2})$
such that $\pi=\begin{bmatrix}
\sigma & \Psi \\
\Psi^* & \rho
\end{bmatrix}$ and $V=W\oplus W'\in\mathcal{B}(H_{2}\oplus H_{1},K_{2}\oplus K_{1})$.
 Therefore $((\rho,K_1,W'),(\Psi,K_2,W))$ is a dilation pair for $(\varphi_2,\Phi_2)$ and $\widetilde{\psi_{1}}\leq_{cp}\begin{bmatrix}
W^* & 0 \\
o  & W'^*
\end{bmatrix}\begin{bmatrix}
\sigma & \Psi \\
\Psi^*  & \rho
\end{bmatrix}\begin{bmatrix}
W & 0 \\
0  & W'
\end{bmatrix}.$
Note that $\pi$ is non-degenerate and by the assumption $\mathcal{E}$ is full,
 therefore $\Psi$ is non-degenerate and $\Psi(\mathcal{E})'=\pi(\mathcal{L}_1(\mathcal{E}))'$ by Remark \ref{remark5.2} and Remark \ref{r2.1}.
 Thus by \cite[Theorem 1.4.2]{Arv},
 there is a unique  $T\oplus S\in\Psi(\mathcal{E})'$ such that
  $0\leq T\oplus S\leq id_{K_{2}\oplus K_{1}}$
   and $$\widetilde{\psi_{1}}=\begin{bmatrix}
W^* & 0 \\
0  & W'^*
\end{bmatrix}\begin{bmatrix}
T & 0 \\
0  & S
\end{bmatrix}\begin{bmatrix}
\sigma & \Psi \\
\Psi^*  & \rho
\end{bmatrix}\begin{bmatrix}
W & 0 \\
0  & W'
\end{bmatrix}=\linebreak\begin{bmatrix}
W^*T\sigma(.)W & W^*T\Psi(.)W' \\
W'^*S\Psi(.)^*W  & W'^*S\rho(.)W'
\end{bmatrix}.$$
Since $\widetilde{\psi_{1}}$ is an extension of $\begin{bmatrix}
id & \Phi_1 \\
\Phi_1^*  & \varphi_1
\end{bmatrix}$, one has $\Phi_1(x)=W^*T\Psi(x)W'$ and $\varphi_1(a)=W'^*S\rho(a)W'$ for all
 $x\in\mathcal{E}$ and $a\in\mathcal{A}.$ But note that $\Psi(\mathcal{E})'$ is a $C^*$-algebra,
  hence $T^{\frac{1}{2}}\oplus S^{\frac{1}{2}}=(T\oplus S)^{\frac{1}{2}}\in\Psi(\mathcal{E})',$
  thus $T\Psi(x)=T^{\frac{1}{2}}T^{\frac{1}{2}}\Psi(x)=T^{\frac{1}{2}}\Psi(x)S^{\frac{1}{2}}.$
  Then $\Phi_1(x)=W^*T^{\frac{1}{2}}\Psi(x)S^{\frac{1}{2}}W'.$

\end{proof}

\subsection*{Acknowledgment}
The research of the first author was in part supported by a grant
from IPM (No. 94470046).

\end{document}